\documentclass[12pt]{article}
\usepackage{amssymb}
\usepackage{amsmath}
\usepackage{latexsym}
\usepackage{mathrsfs}
\usepackage{amsthm}
\usepackage{graphicx}
\usepackage{caption}
\usepackage{hyperref}
\usepackage{color}
\usepackage{enumerate}
\allowdisplaybreaks[1]
\usepackage{comment}

\numberwithin{equation}{section}
\newtheorem{theorem}[equation]{Theorem}

\newtheorem{rem}[equation]{Remark}
\newtheorem{lemma}[equation]{Lemma}
\newtheorem{corollary}[equation]{Corollary}

\title{Weyl-type bounds for Steklov eigenvalues}
\author{Luigi Provenzano\footnote{\scriptsize{EPFL,\,MATHGEOM-FSB,\,Station 8,\,CH-1015 Lausanne,\,Switzerland.\,Email:\,luigi.provenzano@epfl.ch}} and Joachim Stubbe\footnote{\scriptsize{EPFL,\,MATHGEOM-FSB,\,Station 8,\,CH-1015 Lausanne,\,Switzerland.\,Email:\,joachim.stubbe@epfl.ch}}}%

\begin{document}

\newcommand{\rea}{\mathbb{R}}

\maketitle

\noindent
{\bf Abstract:}
We present upper and lower bounds for Steklov eigenvalues for domains in $\mathbb{R}^{N+1}$ with $C^2$ boundary compatible with the Weyl asymptotics. In particular, we obtain sharp upper bounds on Riesz-means and the trace of corresponding Steklov heat kernel. The key result is a comparison of Steklov eigenvalues and Laplacian eigenvalues on the boundary of the domain by applying Pohozaev-type identities on an appropriate tubular neigborhood of the boundary and the min-max principle. Asymptotically sharp bounds then follow from bounds for Riesz-means of Laplacian eigenvalues.

\vspace{11pt}


\vspace{6pt}
%




\section{Introduction.} Let $\Omega\subset\mathbb{R}^{N+1}$ be a bounded domain with boundary $\partial\Omega$. We consider the Steklov eigenvalue problem on $\Omega$:
\begin{equation}\label{Steklov}
\begin{cases}
\Delta u=0, & {\rm in\ }\Omega,\\
\frac{\partial u}{\partial\nu}=\sigma u, & {\rm on\ }\partial\Omega,
\end{cases}
\end{equation}
where $\frac{\partial u}{\partial\nu}=\nabla u\cdot\nu$ denotes the derivative of $u$ in the direction of the outward unit normal $\nu$ to $\partial\Omega$. A classical reference for problem \eqref{Steklov} is \cite{steklov} where it was introduced to describe the stationary heat distribution in a body whose flux through the boundary is proportional to the temperature on the boundary. When $N=1$ problem \eqref{Steklov} can be intepreted as the equation of a free membrane the mass of which is concentrated at the boundary (see \cite{lambertiprovenzano1}). The eigenvalues of problem \eqref{Steklov} can be also seen as the eigenvalues of the Dirichlet-to-Neumann map (see e.g., the survey paper \cite{girouardpolterovich}). We also mention that recently the analogue of the Steklov problem has been introduced for the biharmonic operator as well in \cite{buosoprovenzano} (see also \cite{bcp}).

It is well known that under mild regularity conditions on the boundary $\partial\Omega$ (see e.g., \cite{girouardpolterovich} for a detailed discussion), in particular if $\partial\Omega$ is piecewise $C^1$, problem  \eqref{Steklov} admits an increasing sequence of non-negative eigenvalues of the form
\begin{equation*}
0=\sigma_0<\sigma_1\leq\sigma_2\leq\cdots\nearrow+\infty,
\end{equation*}
where the eigenvalues are repeated according to their multiplicity and satisfy the Weyl asymptotic formula (see \cite{Agr2005})

\begin{equation}\label{steklov-ev-asymptotics}
  \underset{j\to\infty}{\lim}\sigma_jj^{\,-1/N}=2\pi\,B_N^{-1/N}|\partial\Omega|^{-1/N},
\end{equation}
with $|\partial\Omega|$ denoting the $N$-dimensional measure of $\partial\Omega$ and $\displaystyle B_N=\frac{\pi^{N/2}}{\Gamma(1+N/2)}$ being
the volume of the $N$-dimensional unit ball. It is an open problem to find bounds on $\sigma_j$ compatible with the Weyl-limit \eqref{steklov-ev-asymptotics} except when $N=1$ and $\partial\Omega$ is smooth (see \cite{hersch}; see also \cite{colboisgirouard} and the survey article \cite{girouardpolterovich}). The situation is different when we consider the eigenvalue problem for the Laplace-Beltrami operator on $\partial\Omega$, that is
\begin{equation}\label{Laplace-Beltrami}
-\Delta_{\partial\Omega}\, \varphi=\lambda \varphi \quad{\rm on\ }\partial\Omega,
\end{equation}
which for a connected and sufficiently regular $\partial\Omega$ (see Remark \ref{regularityremark}) admits an increasing sequence of non-negative eigenvalues of the form
\begin{equation*}
0=\lambda_0<\lambda_1\leq\lambda_2\leq\cdots\nearrow+\infty,
\end{equation*}
satisfying the Weyl asymptotic formula
\begin{equation}\label{laplace-beltrami-ev-asymptotics}
  \underset{j\to\infty}{\lim}\lambda_jj^{\,-2/N}=(2\pi)^2\,B_N^{-2/N}|\partial\Omega|^{-2/N}
\end{equation}
and Weyl-type bounds of the form (see e.g., \cite{buser}, \cite{colbounding})
\begin{equation}\label{laplace-beltrami-ev-weyl-upper-bound}
  \lambda_j\leq a_{\partial\Omega} +b_Nj^{\,2/N}|\partial\Omega|^{-2/N}
\end{equation}
for some positive constants $a_{\partial\Omega}, b_N$ depending only on the geometry and the dimension of the manifold $\partial\Omega$. We refer to \cite{chavel} for an introduction to eigenvalue problems for the Laplace-Beltrami operator on Riemannian manifolds and to \cite{colbounding,colboisconformal,colboisgirouard,hassan} and to the references therein for a more detailed discussion on upper bounds for the eigenvalues of the Laplacian on manifolds.

The above asymptotic formulas suggest that at least for large $j$ the Steklov eigenvalues $\sigma_j$ are related to the Laplacian eigenvalues $\lambda_j$ approximately via
\begin{equation}\label{steklov-laplace-beltrami-asymptotic-relation}
  \sigma_j\approx \sqrt{\lambda_j}.
\end{equation}
The main result of our paper is a comparison between Steklov and Laplacian eigenvalues for all $j$ compatible with the asymptotic relation \eqref{steklov-laplace-beltrami-asymptotic-relation}.
\begin{theorem}\label{main-theorem} Let $\Omega\subset\mathbb{R}^{N+1}$ be a bounded domain with boundary $\partial\Omega$ of class $C^2$ such that $\partial\Omega$ has only one connected component. Then there exists a constant $c_{\Omega}$ such that for all $j\in\mathbb N$
\begin{equation}\label{steklov-laplace-beltrami-ev-comparison-1}
  \lambda_j\leq \sigma_j^2+2c_{\Omega}\sigma_j, \quad  \sigma_j\leq c_{\Omega}+\sqrt{c_{\Omega}^2+\lambda_j}\,.
\end{equation}
In particular,
\begin{equation}\label{steklov-laplace-beltrami-ev-comparison-2}
 \big| \sigma_j-\sqrt{\lambda_j}\big|\leq 2c_{\Omega}.
\end{equation}
\end{theorem}
The constant $c_{\Omega}$ has the dimension of an inverse length and depends explicitely on the dimension $N$, the maximum of the  mean of the absolute values of the principal curvatures $\kappa_i(x)$, $i=1,\ldots, N$, on $\partial\Omega$ and the maximal possible size $\bar h$ of a suitable tubular neighborhood about $\partial\Omega$. For convex domains $\Omega$ we shall improve the estimates \eqref{steklov-laplace-beltrami-ev-comparison-1} such that they become sharp for all $j$  when $\Omega$ is a ball of radius $R$ and give the exact relation
\begin{equation*}
 \lambda_j=\sigma_j^2+\frac{N-1}{R}\,\sigma_j
\end{equation*}
between Steklov and Laplacian eigenvalues on the $N$-dimensional ball and $N$-dimensional sphere of radius $R$ respectively.
\\
Clearly Theorem \ref{main-theorem} implies Weyl-type estimates for Steklov eigenvalues from the bounds \eqref{laplace-beltrami-ev-weyl-upper-bound} for Laplacian eigenvalues (see Corollary \ref{steklov_upper_bounds_thm}). Combining the sharp Weyl-type estimates for Laplacian eigenvalues on hypersurfaces obtained in \cite{HaSt11} with the estimates of Theorem \ref{main-theorem} we prove the following sharp bound for Riesz means of Steklov eigenvalues:
\begin{theorem}\label{Riesz-mean-theorem} Let $\Omega\subset\mathbb{R}^{N+1}$ be a bounded domain with boundary $\partial\Omega$ of class $C^2$ such that $\partial\Omega$ has only one connected component. Then for all $z\geq 0$
\begin{equation}\label{Steklov-ev-Riesz-mean-upper-bound}
  \sum_{j=0}^{\infty}(z-\sigma_j)_{+}^2\leq \frac{2}{(N+1)(N+2)}\,(2\pi)^{-N}B_N|\partial\Omega| \big(z+c_{\Omega}\big)^{N+2}
\end{equation}
where $c_{\Omega}$ is the constant from Theorem \ref{main-theorem}
\end{theorem}
The estimate \eqref{Steklov-ev-Riesz-mean-upper-bound} is asymptotically sharp since
\begin{equation*}
   \underset{z\to\infty}{\lim}z^{-N-2}\sum_{j=0}^{\infty}(z-\sigma_j)_{+}^2=\frac{2}{(N+1)(N+2)}\,(2\pi)^{-N}B_N|\partial\Omega|
\end{equation*}
according to \eqref{steklov-ev-asymptotics}. Theorem \ref{Riesz-mean-theorem} implies sharp upper bounds on the trace of the associated heat kernel (see Corollary \ref{Steklov-heat-trace-upper-bound-cor}) as well as lower bounds on the eigenvalues (see Corollary \ref{Steklov-lower-bound}).

The present paper is organized as follows: in Section \ref{sec:pre} we recall some properties of the squared distance function from the boundary in a suitable tubular neighborhood of a $C^2$ domain. We exploit these properties in Section \ref{sec:har} in order to obtain estimates of boundary integrals of harmonic functions. In particular, we establish a comparison between the $L^2(\partial\Omega)$ norms of the normal derivative and of the tangential gradient of harmonic functions which is used in Section \ref{sec:comp} together with the min-max principle to prove our main Theorem \ref{main-theorem} and, as a consequence, Weyl-type upper bounds for Steklov eigenvalues. In Section \ref{ex} we consider the case of convex $C^2$ domains for which we refine the estimates \eqref{steklov-laplace-beltrami-ev-comparison-1}, which become sharp in the case of the ball. Finally, in Section \ref{riesz} we prove Theorem \ref{Riesz-mean-theorem} as well as upper bounds on the trace of the Steklov heat kernel and lower bounds on Steklov eigenvalues which turn out to be asymptotically sharp.

\section{The squared distance function from the boundary}\label{sec:pre}

In this section we collect a number of properties of the distance and squared distance functions from the boundary $\partial\Omega$ of a $C^2$ domain of $\mathbb R^{N+1}$ which will be used in the proof of the main result. 

We set
$$
d_0(x):={\rm dist}(x,\partial\Omega).
$$
Let $x\in\partial\Omega$ and let $\nu(x)$ denote the outward unit normal to $\partial\Omega$ at $x$. We have the following characterization of $\nu(x)$ in terms of $d_0(x)$:
\begin{lemma}\label{normal}
Let $\Omega$ be a bounded domain in $\mathbb R^{N+1}$ of class $C^2$. Then for $x\in\partial\Omega$
\begin{equation*}
\nu(x)=-\nabla d_0(x).
\end{equation*}
\end{lemma}
We refer to \cite[Ch.7, Theorem 8.5]{shapes} for the proof of Lemma \ref{normal}. Let $h>0$. The $h$-tubular neighborhood $\omega_h$ of $\partial\Omega$ is defined as
\begin{equation}\label{tubular_n}
\omega_h:=\left\{x\in\Omega:d_0(x)< h\right\}
\end{equation}
We have the following:
\begin{theorem}\label{tubular0}
Let $\Omega$ be a bounded domain in $\mathbb R^{N+1}$ of class $C^2$. Then there exists $h>0$ such that every point in $\omega_h$ has a unique nearest point on $\partial\Omega$.
\end{theorem}
We refer to \cite{krantz} for the proof of Theorem \ref{tubular0} (see also \cite[Ch.6, Theorem 6.3]{shapes} and \cite[Lemma 14.16]{gitr}). Throughout the rest of the paper we shall denote by $\bar h$ the maximal possible tubular radius of $\Omega$, namely
\begin{equation}\label{barh}
\bar h:=\sup\left\{h>0:{\rm every\  point\  in\ } \omega_{h} {\rm\ has\ a\ unique\ nearest\ point\ on\ } \partial\Omega\right\}.
\end{equation}
From Theorem \ref{tubular0} it follows that if $\Omega$ is of class $C^2$ such $\bar h$ exists and is positive. For any $h\in]0,\bar h[$ we denote by $\Gamma_h$ the set
\begin{equation}\label{Gamma}
\Gamma_h:=\partial\omega_{{h}}\setminus\partial\Omega.
\end{equation}
Throughout the rest of this section, we will denote by $h$ a positive number such that $h\in]0,\bar h[$. In a tubular neighborhood $\omega_h$ the distance function (and hence its square) is of class $C^2$. This is stated in the following:

\begin{theorem}\label{tubular}
Let $\Omega$ be a bounded domain in $\mathbb R^{N+1}$ of class $C^2$. Let $\omega_{{h}}$ be as in \eqref{tubular_n}. Then $d_0$ is of class $C^2$ in $\omega_h$. Moreover, for any $x\in\partial\Omega$, the matrix $D^2\left(d_0(x)^2/2\right)$ represents the orthogonal projection on the normal space to $\partial\Omega$ at $x$ and
$$
d_0(x-p\nu(x))=|p|,
$$
$$
\nabla d_0(x-p\nu(x))=-\nu(x),
$$
for any $p\in\mathbb R$ with $|p|\leq h$.
\end{theorem}

We refer to \cite[Theorem 3.1]{ambrosio1}, \cite[Ch.7, Theorem 8.5]{shapes} and \cite[Lemma 14.16]{gitr} for the proof of Theorem \ref{tubular}. The situation described in Theorems \ref{tubular0} and \ref{tubular} is illustrated in Figure \ref{F1}.

\begin{figure}[ht]
\centering
\includegraphics[width=0.5\textwidth]{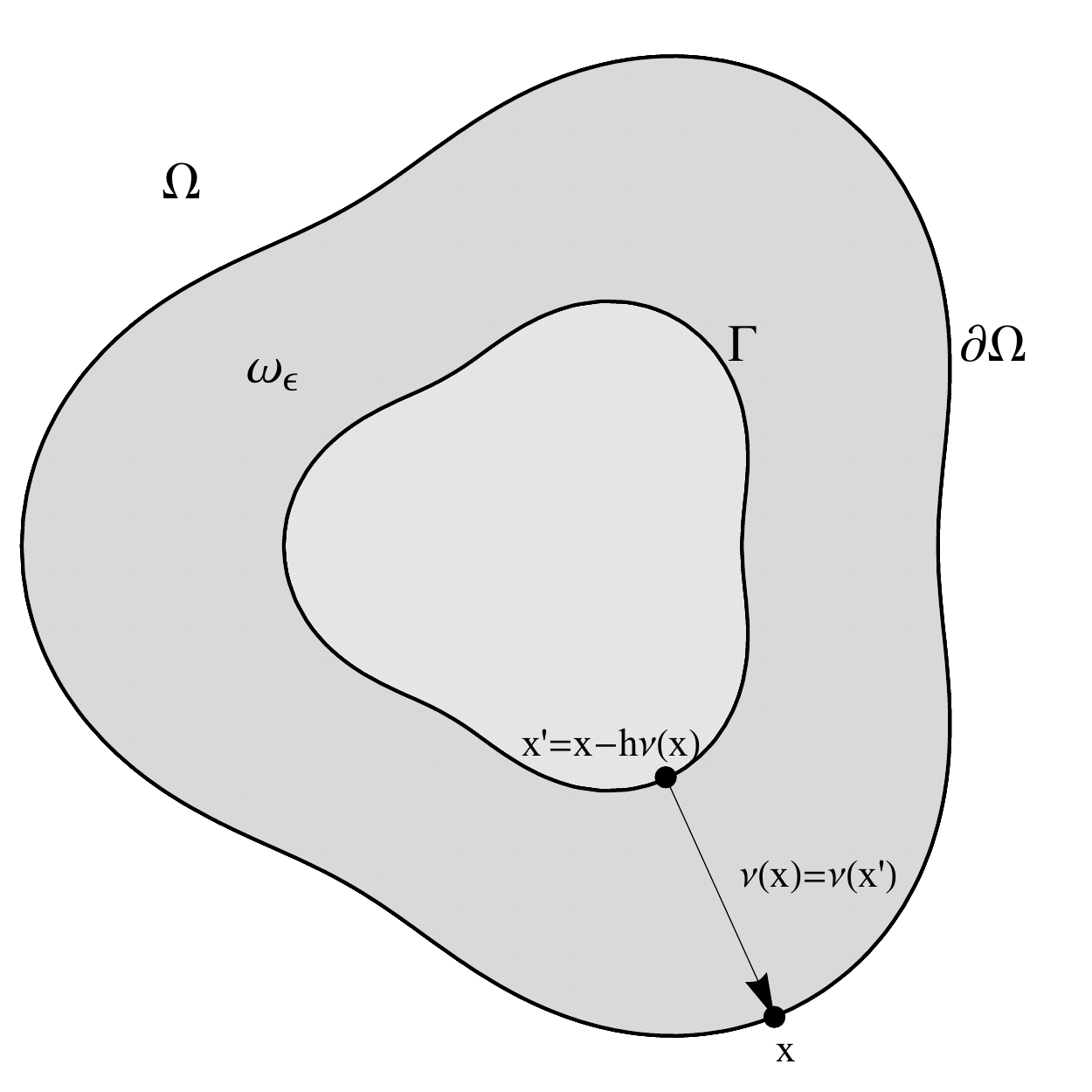}
\caption{Tubular neighborhood of a $C^2$ planar domain.}
\label{F1}
\end{figure}

\begin{rem}
From Theorem \ref{tubular} it follows that the set $\Gamma_h$ is diffeomorphic to $\partial\Omega$.
\end{rem}

Let $x\in\partial\Omega$ and let $\kappa_1(x),...,\kappa_{N}(x)$ denote the principal curvatures of $\partial\Omega$ at $x$ with respect to the outward unit normal. We refer e.g., to \cite[Sec. 14.6]{gitr} for the definition and basic properties of the principal curvatures of $\partial\Omega$. We have the following:
\begin{lemma}\label{bound_curv_lemma}
Let $\Omega$ be a bounded domain in $\mathbb R^{N+1}$ of class $C^2$. Let $x\in\omega_{{h}}$ and let $y\in\partial\Omega$ be the nearest point to $x$ on $\partial\Omega$. Then
\begin{equation}\label{bound_curv}
1-d_0(x)\kappa_i(y)>0
\end{equation}
for all $i=1,...,N$.
\end{lemma}
We refer to \cite[Lemma 2.2]{lewis} for a proof of Lemma \ref{bound_curv_lemma}. We note that the number ${\bar h}$ in \eqref{barh} provides an upper bound for the positive principal curvatures of $\partial\Omega$. In fact we have
\begin{equation}\label{K1}
K_+:=\max_{\substack{1\leq i\leq N,\\x\in\partial\Omega}}\max\left\{0,\kappa_i(x)\right\}<\frac{1}{{\bar h}}.
\end{equation}
We also define $K_-$ by
\begin{equation}\label{K2}
K_-:=\min_{\substack{1\leq i\leq N,\\x\in\partial\Omega}}\min\left\{0,\kappa_i(x)\right\}\leq 0.
\end{equation}
and $K_{\infty}$ by
\begin{equation}\label{KKKK}
K_{\infty}:=\max\left\{K_+,-K_-\right\}=\max_{\substack{1\leq i\leq N,\\x\in\partial\Omega}}|\kappa_i(x)|.
\end{equation}

Now we introduce the functions ${d}$ and ${\eta}$ from $\omega_{{h}}$ to $\mathbb R$ defined by
$$
{d(x)}:={\rm{dist}}(x,\Gamma_h)
$$
and
$$
\eta(x):=\frac{{d(x)}^2}{2}.
$$
Clearly ${d(x)}={h}-d_0(x)$ for all $x\in\omega_{{h}}$, hence $d$ and $\eta$ are of class $C^2$ in $\omega_{{h}}$.

Let $x\in\partial\Omega$ and let ${x'}=x-{h}\nu(x)\in\Gamma_h$. Let now ${\kappa}_1'({x'}),...,{\kappa}_{N}'({x'})$ denote the principal curvatures of $\Gamma_h$ at ${x'}$ with respect to the outward unit normal. The principal curvatures ${\kappa}_i'({x'})$ and $\kappa_i(x)$ are related, as stated in the following:
\begin{lemma}\label{linklemma}
Let $\Omega$ be a bounded domain in $\mathbb R^{N+1}$ of class $C^2$. Let $\omega_{{h}}$ and $\Gamma_h$ be defined by \eqref{tubular_n} and \eqref{Gamma}, respectively. Let $x\in\partial\Omega$ and let ${x'}=x-{h}\nu(x)\in\Gamma_h$. Then we have
\begin{equation}\label{link}
{\kappa}_i'({x'})=\frac{\kappa_i(x)}{1-{h} \kappa_i(x)}
\end{equation}
for all $i=1,...,N$. Moreover, $\nu(x)=\nu({x'})$.
\end{lemma}
The proof of Lemma \ref{linklemma} follows from \cite[Theorem 3]{ambrosiobook} and from the fact that ${d(x)}={h}-d_0(x)$ (see also \cite{Sakaguchi2016}).


Now we are ready to state the following theorem concerning the eigenvalues of $D^2{\eta}$.

\begin{theorem}\label{eigenvalues_dist_sq}
Let $\Omega$ be a bounded domain in $\mathbb R^{N+1}$ of class $C^2$. Let $\omega_{{h}}$ and $\Gamma_h$ be defined by \eqref{tubular_n} and \eqref{Gamma}, respectively. Let $x\in\omega_{{h}}$ and let ${y'}=x+{d(x)}\nabla{d(x)}\in \Gamma_h$ be the nearest point to $x$ on $\Gamma_h$. Then, denoting by $\rho_1(x),...,\rho_N(x)$ the eigenvalues of $D^2{\eta}(x)$ it holds
\begin{equation*}
\rho_i(x)=\begin{cases}
\frac{{d(x)}\kappa_i'({y'})}{1+{d(x)}{\kappa}_i'({y'})}, &  {\rm if\ } 1\leq i\leq N,\\
1, & {\rm if\ } i=N+1.
\end{cases}
\end{equation*}
\end{theorem}
The proof of Theorem \ref{eigenvalues_dist_sq} can be carried out in a similar way as in \cite[Lemma 1]{balinsky} (see also \cite[Lemma 14.17]{gitr}). We also refer to \cite[Theorem 4]{ambrosiobook} and \cite[Theorem 3.2]{ambrosio1} for an alternative approach.

From now on we will agree to order the eigenvalues $\rho_i(x)$ of $D^2{\eta}(x)$ increasingly, so that $\rho_1(x)\leq\rho_2(x)\leq\cdots\leq\rho_{N+1}(x)=1$.

We conclude this section by presenting some bounds for the eigenvalues $\rho_i(x)$ when $x\in\omega_{{h}}$. We have the following:
\begin{lemma}
Let $\Omega$, $\omega_{{h}}$ and $\Gamma_h$ be as in Theorem \ref{eigenvalues_dist_sq}. Let $x\in\omega_{{h}}$ and let $\rho_i(x)$ denote the eigenvalues of $D^2{\eta}(x)$ for $i=1,...,N$. Then
\begin{equation}\label{eig_size_eq}
{h}{K_-}\leq\rho_i(x)\leq {h} K_+<1.
\end{equation}
\proof
 Let $x\in\omega_{{h}}$ and let $y$ be the unique nearest point to $x$ on $\partial\Omega$. From \eqref{link} and from the fact that ${d(x)}={h}-d_0(x)$ it follows that
\begin{equation}\label{rho_kappa}
\rho_i(x)=1-\frac{1-{h}\kappa_i(y)}{1-d_0(x)\kappa_i(y)}.
\end{equation}
We observe that the function $\kappa\mapsto 1-\frac{1-{h}\kappa}{1-d\kappa}$ is increasing and convex for all $0\leq d\leq{h}$, provided $\kappa<1/{h}$ (which is always the case, see \eqref{K1} and \eqref{K2}). Moreover the function $d\mapsto 1-\frac{1-{h}\kappa}{1-d\kappa}$ is decreasing and concave if $\kappa\geq 0$ and increasing and concave if $\kappa\leq 0$. Then
$$
\rho_i(x)\leq 1-\frac{1-{h} K_+}{1-d_0(x)K_+}\leq{h} K_+
$$
and
$$
\rho_i(x)\geq 1-\frac{1-{h} K_-}{1-d_0(x)K_-}\geq {h}K_-,
$$
since $K_-\leq 0\leq K_+$.
This concludes the proof.
\endproof
\end{lemma}

\begin{rem}\label{convex}
If $\Omega$ is a convex domain of class $C^2$ we have that $\kappa_i(x)\geq 0$ for all $i=1,...,N$ and for all $x\in\partial\Omega$, hence $0\leq\rho_i(x)\leq 1$, for all $i=1,...,N+1$ and for all $x\in\omega_{{h}}$. Moreover Theorem \ref{eigenvalues_dist_sq} holds for all $h\in]0,1/K_{\infty}[$ (see Section \ref{ex}). This is not true for general non-convex domains, since it is not possible to estimate the size of the maximum tubular neighborhood $\omega_{{h}}$ only in terms of the principal curvatures. In fact ${h}$ can be much smaller than $1/K_{\infty}$ (see Figure \ref{F2}).
\end{rem}

\begin{figure}[ht]
\centering
\includegraphics[width=0.5\textwidth]{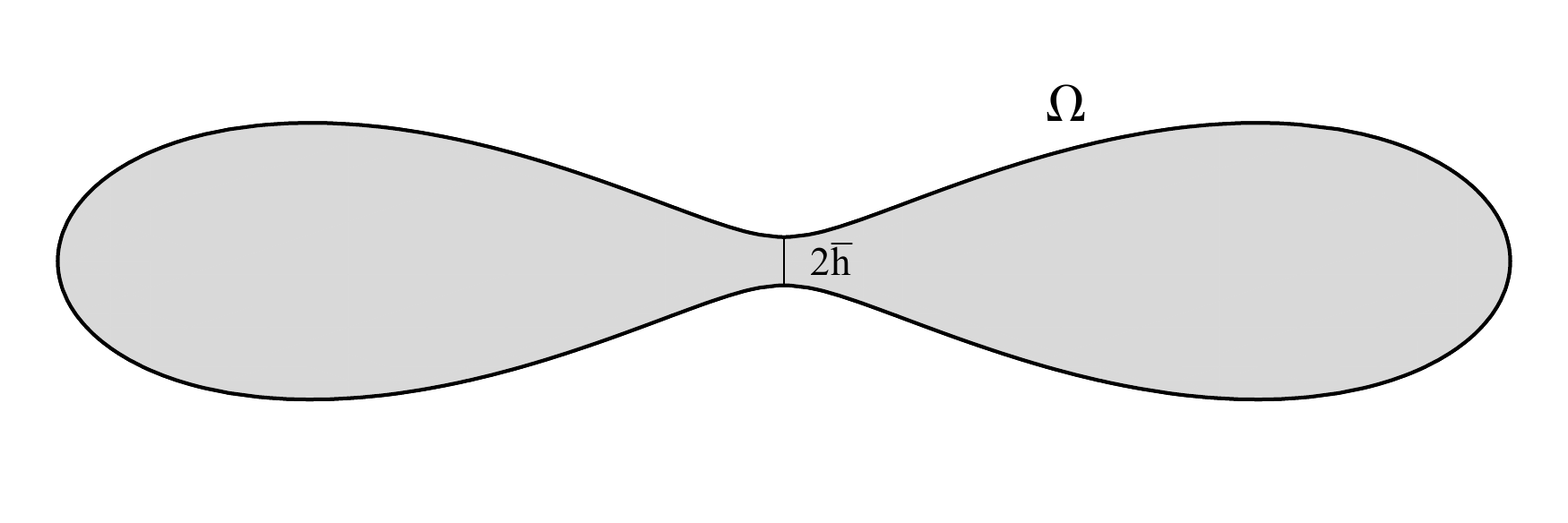}
\caption{If the domain is not convex we can have arbitrary small ${\bar h}$ while $K_{\infty}$ is uniformly bounded.}
\label{F2}
\end{figure}

\section{Boundary integrals of harmonic functions}\label{sec:har}

The aim of this section is to prove that for a function $v\in H^2(\Omega)$ harmonic in $\Omega$, the norms $\|\nabla_{\partial\Omega}v\|_{L^2(\partial\Omega)}$ and $\|\frac{\partial v}{\partial\nu}\|_{L^2(\partial\Omega)}$ are equivalent. Here $\nabla_{\partial\Omega}v$ denotes the tangential gradient of a function $v\in H^1(\partial\Omega)$. This is the usual intrinsic gradient of $v$ on the Riemannian $C^2$-manifold $\partial\Omega$ with the induced Riemannian metric of $\mathbb R^{N+1}$. 
We will denote by $H^m(\Omega)$ (respectively $H^m(\partial\Omega)$) the Sobolev spaces of real-valued functions in $L^2(\Omega)$ (respectively $L^2(\partial\Omega)$) with weak derivatives up to order $m$ in $L^2(\Omega)$ (respectively $L^2(\partial\Omega)$). We will also denote by $d\sigma$ the $N$-dimensional measure element of $\partial\Omega$.

We start with the following generalized Pohozaev identity for harmonic functions:
\begin{lemma}\label{poholemma}
Let $F:\Omega\rightarrow\mathbb R^{N+1}$ be a Lipschitz vector field. Let $v\in H^2(\Omega)$ with $\Delta v=0$ in $\Omega$. Then
\begin{multline}\label{pohozaev}
\int_{\partial\Omega}\frac{\partial v}{\partial\nu}F\cdot\nabla v d\sigma-\frac{1}{2}\int_{\partial\Omega}|\nabla v|^2F\cdot\nu d\sigma\\
+\frac{1}{2}\int_{\Omega}|\nabla v|^2 {\rm div}F dx-\int_{\Omega}(DF\cdot\nabla v)\cdot\nabla v dx=0,
\end{multline}
where $DF$ denotes the Jacobian matrix of $F$.
\proof
Since $v$ is harmonic in $\Omega$, we have $\Delta vF\cdot\nabla v=0$ in $\Omega$. We integrate such identity over $\Omega$. Throughout the rest of the proof we shall write $\partial_i v$ for $\frac{\partial v}{\partial x_i}$ and $\partial^2_{ik} v$ for $\frac{\partial^2 v}{\partial x_i\partial x_k}$. We have
\begin{multline}\label{poho_proof}
0=\int_{\Omega}\Delta v F\cdot\nabla v dx=\int_{\partial\Omega}\frac{\partial v}{\partial\nu}F\cdot\nabla v d\sigma-\int_{\Omega}\nabla v\cdot\nabla(F\cdot\nabla v)dx\\
=\int_{\partial\Omega}\frac{\partial v}{\partial\nu}F\cdot\nabla v d\sigma-\int_{\Omega}(DF\cdot\nabla v)\cdot\nabla v dx-\int_{\Omega}(D^2v\cdot F)\cdot\nabla v dx,
\end{multline}
where $D^2v$ denotes the Hessian matrix of $v$. Now let us consider the third summand in \eqref{poho_proof}. We have
\begin{multline*}
\int_{\Omega}(D^2v\cdot F)\cdot\nabla v dx=\int_{\Omega}\sum_{i,k=1}^{N+1}\partial_iv\partial^2_{ik}v F_k dx\\
=\int_{\partial\Omega}\sum_{i,k=1}^{N+1}\partial_iv\partial_ivF_k\nu_kd\sigma-\int_{\Omega}\sum_{i,k=1}^{N+1}\partial_iv\partial_k(\partial_iv F_k)dx\\
=\int_{\partial\Omega}|\nabla u|^2 F\cdot\nu d\sigma-\int_{\Omega}|\nabla v|^2{\rm div}F dx-\int_{\Omega}(D^2v\cdot F)\cdot\nabla v dx,
\end{multline*}
thus
\begin{equation}\label{piece}
\int_{\Omega}(D^2v\cdot F)\cdot\nabla v dx=\frac{1}{2}\int_{\partial\Omega}|\nabla v|^2 F\cdot\nu d\sigma-\frac{1}{2}\int_{\Omega}|\nabla v|^2{\rm div}F dx.
\end{equation}
We plug \eqref{piece} in \eqref{poho_proof} and finally obtain \eqref{pohozaev}. This concludes the proof of the lemma.
\endproof
\end{lemma}

\begin{rem}
When $F=x$, formula \eqref{pohozaev} is usually referred as Pohozaev identity. It reads
\begin{equation}\label{poho0}
\int_{\partial\Omega}\frac{\partial v}{\partial\nu}x\cdot\nabla v d\sigma-\frac{1}{2}\int_{\partial\Omega}|\nabla v|^2x\cdot\nu d\sigma+\frac{N-1}{2}\int_{\Omega}|\nabla v|^2 dx=0,
\end{equation}
for all $v\in H^2(\Omega)$ with $\Delta v=0$. Formula \eqref{poho0} when $\Omega$ is a ball in $\mathbb R^{N+1}$ allows to write the exact relations between the Steklov eigenvalues of $\Omega$ and the Laplace-Beltrami eigenvalues on $\partial\Omega$ without knowing explicitly the eigenvalues (see Subsection \ref{balls}). For a general domain $\Omega$ of class $C^2$ it is natural to use $F$ as in \eqref{field} here below.
\end{rem}

Let $\Omega$ be a bounded domain of class $C^2$ in $\mathbb R^{N+1}$. Let ${h}\in]0,\bar h[$, where $\bar h$ is given by \eqref{barh}, and $\omega_{{h}}$ be as in \eqref{tubular_n}. Let $F:\Omega\rightarrow\mathbb R^{N+1}$ be defined by
\begin{equation}\label{field}
F(x):=\begin{cases}
0,& {\rm if\ }x\in\Omega\setminus\omega_{{h}},\\
\nabla{\eta}, & {\rm if\ }x\in\omega_{{h}}.
\end{cases}
\end{equation}
By construction $F$ is a Lipschitz vector field. We consider formula \eqref{pohozaev} with $F$ given by \eqref{field}. We use the fact that for $v\in H^1(\Omega)$ (and hence for $v\in H^2(\Omega)$), $|\nabla v|^2_{|_{\partial\Omega}}=|\nabla_{\partial\Omega}v|^2+\left(\frac{\partial v}{\partial\nu}\right)^2$. Moreover, we use the fact that $F(x)={h}\nu(x)$ when $x\in\partial\Omega$. We have
\begin{multline}\label{main}
0=\int_{\partial\Omega}\left(\frac{\partial v}{\partial\nu}\right)^2d\sigma-\int_{\partial\Omega}|\nabla_{\partial\Omega}v|^2d\sigma\\
+\frac{1}{{h}}\left(\int_{\omega_{{h}}}|\nabla v|^2\Delta{\eta}-2(D^2{\eta}\cdot\nabla v)\cdot\nabla v dx\right).
\end{multline}
Let $x\in\omega_{{h}}$. From \eqref{main}, in order to compare the integrals of $|\nabla_{\partial\Omega}v|^2$ and $\left(\frac{\partial v}{\partial\nu}\right)^2$ over $\partial\Omega$, we have to estimate
\begin{equation}\label{integral}
|\nabla v(x)|^2\Delta{\eta}(x)-2(D^2{\eta}(x)\cdot\nabla v(x))\cdot\nabla v(x).
\end{equation}
We have the following lemma.
\begin{lemma}\label{estimate_big}
Let $\Omega$ be a bounded domain in $\mathbb R^{N+1}$ of class $C^2$. Let $\omega_{{h}}$ be as in \eqref{tubular_n}. For any $v\in H^1(\Omega)$ it holds
\begin{equation}\label{estimate}
\left|\int_{\omega_{{h}}}|\nabla v|^2\Delta{\eta}-2(D^2{\eta}\cdot\nabla v)\cdot\nabla v dx\right|\leq \left(1+N{\bar{H}_{\infty}}{h}\right)\int_{\Omega}|\nabla v|^2 dx,
\end{equation}
where 
\begin{equation*}
{\bar{H}_{\infty}}:=\max_{x\in\partial\Omega}\left(\frac{1}{N}\sum_{i=1}^N|\kappa_i(x)|\right).
\end{equation*}
\proof
Let $x\in\omega_{{h}}$. Let $\xi_i(x)$, $i=1,...,N+1$ be the eigenvectors of $D^2{\eta}(x)$ associated with the eigenvalues $\rho_i(x)$ and normalized such that $\xi_i(x)\cdot\xi_j(x)=\delta_{ij}$. We can write then 
$$
\nabla v(x)=\sum_{i=1}^{N+1}\alpha_i(x)\xi_i(x),
$$
for some $\alpha_i(x)\in\mathbb R$. We note that $|\nabla v(x)|^2=\sum_{i=1}^{N+1}\alpha_i(x)^2$. With this notation \eqref{integral} can be re-written as follows:
\begin{multline}\label{integral2}
Q(\nabla v(x)):=|\nabla v(x)|^2\Delta{\eta}(x)-2(D^2{\eta}(x)\cdot\nabla v(x))\cdot\nabla v(x)\\=\sum_{i=1}^{N+1}\alpha_i(x)^2\sum_{i=1}^N\rho_i(x)-2\sum_{i=1}^{N+1}\rho_i(x)\alpha_i(x)^2.
\end{multline}
Suppose that $\nabla v\ne 0$, otherwise inequality \eqref{estimate} is trivially true. We have that
\begin{equation}\label{QQ}
Q(\nabla v(x))=\sum_{i=1}^{N+1}\rho_i(x)(1-2\tilde\alpha_i(x)^2)|\nabla v(x)|^2,
\end{equation}
where
$$
\tilde\alpha_i(x):=\frac{\alpha_i(x)}{\sqrt{\sum_{i=1}^{N+1}\alpha_i(x)^2}}=\frac{\alpha_i(x)}{|\nabla v(x)|}.
$$
It is straightforward to see that
\begin{multline}\label{maxemin}
\left(\sum_{i=1}^{N}\rho_i(x)-1\right)|\nabla v(x)|^2\leq\sum_{i=1}^{N+1}\rho_i(x)(1-2\tilde\alpha_i(x)^2)|\nabla v(x)|^2\\
\leq \left(1+\sum_{i=2}^{N}\rho_i(x)-\rho_1(x)\right)|\nabla v(x)|^2.
\end{multline}
Now from \eqref{bound_curv},\eqref{eig_size_eq} and \eqref{rho_kappa} it follows that
\begin{equation*}
|\rho_i(x)|-h|\kappa_i(x)|=-\frac{d_0(x)|\kappa_i(y)|(1-\kappa_i(y))}{1-d_0(x)\kappa_i(y)}\leq 0
\end{equation*}
for all $x\in\omega_h$ and $i=1,...,N$, where $y\in\partial\Omega$ is the unique nearest point to $x$ on $\partial\Omega$. Hence for all $x\in\omega_h$
\begin{equation}\label{max}
\sum_{i=2}^{N}\rho_i(x)-\rho_1(x)\leq\sum_{i=1}^N|\rho_i(x)|\leq N\bar{H}_{\infty}h.
\end{equation}
On the other hand, again from \eqref{bound_curv},\eqref{eig_size_eq} and \eqref{rho_kappa} we have that
\begin{equation*}
\rho_i(x)-(h-d_0(x))\kappa_i(y)=\frac{(h-d_0(x))d_0(x)\kappa_i(y)^2}{1-d_0(x)\kappa_i(y)}\geq 0,
\end{equation*}
for all $x\in\omega_h$ and $i=1,...,N$, where $y\in\partial\Omega$ is the unique nearest point to $x$ on $\partial\Omega$. Hence for all $x\in\omega_h$
\begin{equation}\label{min}
\sum_{i=1}^N\rho_i(x)\geq -NH_{\infty}h\geq-N\bar{H}_{\infty}h,
\end{equation}
where
\begin{equation}\label{HHH}
{{H}_{\infty}}:=\max_{x\in\partial\Omega}\frac{1}{N}\left|\sum_{i=1}^N\kappa_i(x)\right|
\end{equation}
denotes the maximal mean curvature of $\partial\Omega$. Clearly \eqref{integral2}, \eqref{QQ}, \eqref{maxemin},\eqref{max} and \eqref{min} imply
$$
|Q(\nabla v(x))|\leq (1+N\bar{H}_{\infty}{h})|\nabla v(x)|^2
$$
and therefore the validity of \eqref{estimate}. This concludes the proof.

\endproof
\end{lemma}

From now on we will write 
\begin{equation*}
c_{\Omega}:=\frac{1}{2\bar h}+\frac{N\bar{H}_{\infty}}{2},
\end{equation*}
where $\bar h$ is given by \eqref{barh}. We are ready to prove the following theorem:
\begin{theorem}\label{equivalence}
Let $v\in H^2(\Omega)$ be such that $\Delta v=0$ in $\Omega$ and normalized such that $\int_{\partial\Omega}v^2 d\sigma=1$. Then it holds
\begin{enumerate}[i)]
\item 
\begin{equation}\label{estimate_tangential}
\int_{\partial\Omega}|\nabla_{\partial\Omega}v|^2d\sigma\leq \int_{\partial\Omega}\left(\frac{\partial v}{\partial\nu}\right)^2d\sigma+{2c_{\Omega}}\left(\int_{\partial\Omega}\left(\frac{\partial v}{\partial\nu}\right)^2d\sigma\right)^{\frac{1}{2}};
\end{equation}
\item
\begin{equation}\label{estimate_normal}
\left(\int_{\partial\Omega}\left(\frac{\partial v}{\partial\nu}\right)^2d\sigma\right)^{\frac{1}{2}}\leq{c_{\Omega}}+\sqrt{{c_{\Omega}^2}+\int_{\partial\Omega}|\nabla_{\partial\Omega}v|^2d\sigma}.
\end{equation}
\end{enumerate}
\proof
Let $h\in]0,\bar h[$. We start by proving $i)$. From Lemmas \ref{poholemma} and \ref{estimate_big} we have
\begin{multline}\label{chain1}
\int_{\partial\Omega}|\nabla_{\partial\Omega}v|^2d\sigma= \int_{\partial\Omega}\left(\frac{\partial v}{\partial\nu}\right)^2d\sigma+\frac{1}{{h}}\left(\int_{\omega_{{h}}}|\nabla v|^2\Delta{\eta}-2(D^2{\eta}\cdot\nabla v)\cdot\nabla v dx\right)\\
\leq \int_{\partial\Omega}\left(\frac{\partial v}{\partial\nu}\right)^2d\sigma+\left(\frac{1}{h}+N\bar{H}_{\infty}\right)\int_{\Omega}|\nabla v|^2dx\\
=\int_{\partial\Omega}\left(\frac{\partial v}{\partial\nu}\right)^2d\sigma+\left(\frac{1}{h}+N\bar{H}_{\infty}\right)\int_{\partial\Omega}v\frac{\partial v}{\partial\nu}dx\\
\leq \int_{\partial\Omega}\left(\frac{\partial v}{\partial\nu}\right)^2d\sigma+\left(\frac{1}{h}+N\bar{H}_{\infty}\right)\left(\int_{\partial\Omega}v^2d\sigma\right)^{\frac{1}{2}}\left(\int_{\partial\Omega}\left(\frac{\partial v}{\partial\nu}\right)^2d\sigma\right)^{\frac{1}{2}}\\
=\int_{\partial\Omega}\left(\frac{\partial v}{\partial\nu}\right)^2d\sigma+\left(\frac{1}{h}+N\bar{H}_{\infty}\right)\left(\int_{\partial\Omega}\left(\frac{\partial v}{\partial\nu}\right)^2d\sigma\right)^{\frac{1}{2}},
\end{multline}
where we have used the following Green's identity
$$
\int_{\Omega}\Delta v v dx=\int_{\partial\Omega}\frac{\partial v}{\partial\nu} v d\sigma-\int_{\Omega}|\nabla v|^2 dx,
$$
the fact that $\Delta v=0$ in $\Omega$, H\"older's inequality and the fact that $\int_{\partial\Omega}v^2d\sigma=1$. Since \eqref{chain1} holds true for all $h\in]0,\bar h[$, it is true with $h=\bar h$. This proves $i)$. We repeat a similar argument for $ii)$. We have
\begin{multline}\label{lower0}
\int_{\partial\Omega}\left(\frac{\partial v}{\partial\nu}\right)^2d\sigma=\int_{\partial\Omega}|\nabla_{\partial\Omega}v|^2d\sigma-\frac{1}{{h}}\left(\int_{\omega_{{h}}}|\nabla v|^2\Delta{\eta}-2(D^2{\eta}\cdot\nabla v)\cdot\nabla v dx\right)\\
\leq \int_{\partial\Omega}|\nabla_{\partial\Omega}v|^2d\sigma+\left(\frac{1}{h}+N\bar{H}_{\infty}\right)\left(\int_{\partial\Omega}\left(\frac{\partial v}{\partial\nu}\right)^2d\sigma\right)^{\frac{1}{2}},
\end{multline}
which is equivalent to
$$
\int_{\partial\Omega}\left(\frac{\partial v}{\partial\nu}\right)^2d\sigma-\left(\frac{1}{h}+N\bar{H}_{\infty}\right)\left(\int_{\partial\Omega}\left(\frac{\partial v}{\partial\nu}\right)^2d\sigma\right)^{\frac{1}{2}}-\int_{\partial\Omega}|\nabla_{\partial\Omega}v|^2d\sigma\leq 0.
$$
This is an inequality of degree two in the unknown $\left(\int_{\partial\Omega}\left(\frac{\partial v}{\partial\nu}\right)^2d\sigma\right)^{\frac{1}{2}}\geq 0$. Solving the inequality we obtain
\begin{equation}\label{chain2}
\left(\int_{\partial\Omega}\left(\frac{\partial v}{\partial\nu}\right)^2d\sigma\right)^{\frac{1}{2}}\leq\left(\frac{1}{2h}+\frac{N\bar{H}_{\infty}}{2}\right)+\sqrt{\left(\frac{1}{2h}+\frac{N\bar{H}_{\infty}}{2}\right)^2+\int_{\partial\Omega}|\nabla_{\partial\Omega}v|^2d\sigma}.
\end{equation}
Since \eqref{chain2} holds true for all $h\in]0,\bar h[$, it is true with $h=\bar h$. This concludes the proof of $ii)$ and of the theorem.
\endproof
\end{theorem}

Theorem \ref{equivalence} states that for harmonic functions $v$ in $\Omega$ the $L^2(\partial\Omega)$ norms of $\frac{\partial v}{\partial\nu}$ and of $\nabla_{\partial\Omega}v$ are equivalent. This will be used in the next section to compare the Steklov eigenvalues on $\Omega$ with the Laplace-Beltrami eigenvalues on $\partial\Omega$.

\begin{rem}
We note that thanks to \eqref{maxemin} and \eqref{min} we can use the maximal mean curvature $H_{\infty}$ instead of $\bar{H}_{\infty}$ in \eqref{lower0} and therefore in the inequality \eqref{chain2}. Moreover, this and inequality \eqref{chain2} imply
\begin{equation}\label{bound_mean_curv}
\int_{\Omega}|\nabla v|^2dx\leq \left(\frac{1}{2h}+\frac{N{H}_{\infty}}{2}\right)+\sqrt{\left(\frac{1}{2h}+\frac{N{H}_{\infty}}{2}\right)^2+\int_{\partial\Omega}|\nabla_{\partial\Omega}v|^2d\sigma}.
\end{equation}
for all $v\in H^2(\Omega)$ with $\Delta v=0$ and $\int_{\partial\Omega}v^2d\sigma=1$.
\end{rem}

\section{Proof of Theorem \ref{main-theorem}}\label{sec:comp}

In this section we prove Theorem \ref{main-theorem}. Namely, we prove that the absolute value of the difference between the $j$-th eigenvalues of problems \eqref{Steklov} and \eqref{Laplace-Beltrami} is bounded by $2c_{\Omega}$. Throughout the rest of the paper we shall assume that $\Omega$ is a bounded domain of class $C^2$ in $\mathbb R^{N+1}$ such that its boundary $\partial\Omega$ has only one connected component. This says that $\partial\Omega$ is a compact $C^2$-submanifold of dimension $N$ in $\mathbb R^{N+1}$ without boundary. In particular, $\partial\Omega$ is a Riemannian $C^2$-manifold of dimension $N$ with the induced Riemannian metric.

The proof of Theorem \ref{main-theorem} is carried out by exploiting Theorem \ref{equivalence} and the following variational characterizations of the eigenvalues of problems \eqref{Steklov} and \eqref{Laplace-Beltrami}, namely

\begin{equation}\label{steklov_minmax}
\sigma_j=\inf_{\substack{V\leq \tilde H^1(\Omega),\\{\rm dim}V=j}}\sup_{\substack{0\ne v\in V,\\\int_{\partial\Omega}v^2d\sigma=1}}\int_{\Omega}|\nabla v|^2dx,
\end{equation}
for all $j\in\mathbb N$, $j\geq 1$, where
$$
\tilde H^1(\Omega):=\left\{v\in H^1(\Omega):\int_{\partial\Omega}vd\sigma=0\right\},
$$
and
\begin{equation}\label{LB_minmax}
\lambda_j=\inf_{\substack{V\leq \tilde H^1(\partial\Omega),\\{\rm dim}V=j}}\sup_{\substack{0\ne v\in V,\\\int_{\partial\Omega}v^2d\sigma=1}}\int_{\partial\Omega}|\nabla_{\partial\Omega} v|^2d\sigma,
\end{equation}
for all $j\in\mathbb N$, $j\geq 1$, where
$$
\tilde H^1(\partial\Omega):=\left\{v\in H^1(\partial\Omega):\int_{\partial\Omega}vd\sigma=0\right\}.
$$

It is useful to recall the following results on the completeness of the sets of eigenfunctions of problems \eqref{Steklov} and \eqref{Laplace-Beltrami} in $L^2(\partial\Omega)$.

\begin{theorem}\label{LBort}
Let $\Omega$ be a bounded domain in $\mathbb R^{N+1}$ of class $C^2$. Let $\left\{\sigma_j\right\}_{j=0}^{\infty}$ be the sequence of eigenvalues of problem \eqref{Steklov} and let $\left\{u_j\right\}_{j=0}^{\infty}\subset H^1(\Omega)$ denote the sequence of eigenfunctions associated with the eigenvalues $\sigma_j$, normalized such that $\int_{\partial\Omega}u_iu_k d\sigma=\delta_{ik}$ for all $i,k\in\mathbb N$. Then $\left\{{u_j}{|_{\partial\Omega}}\right\}_{j=0}^{\infty}$ is an orthonormal basis of $L^2(\partial\Omega)$. Moreover, $\int_{\Omega}\nabla u_i\cdot\nabla u_k dx=\sigma_i\delta_{ik}$ for all $i,k\in\mathbb N$.
\end{theorem}
We refer e.g., to \cite{auchbases} for a proof of Theorem \ref{LBort} (see also \cite{brezis,daviesbook}). 






\begin{theorem}\label{LBeigen}
Let $\Omega$ be a bounded domain in $\mathbb R^{N+1}$ of class $C^2$. Let $\left\{\lambda_j\right\}_{j=0}^{\infty}$ be the sequence of eigenvalues of problem \eqref{Laplace-Beltrami} and let $\left\{\varphi_j\right\}_{j=0}^{\infty}\subset H^1(\partial\Omega)$ denote the sequence of eigenfunctions associated with the eigenvalues $\lambda_j$, normalized such that $\int_{\partial\Omega}\varphi_i\varphi_k d\sigma=\delta_{ik}$ for all $i,k\in\mathbb N$.
 Then $\left\{\varphi_j\right\}_{j=0}^{\infty}$ is an orthonormal basis of $L^2(\partial\Omega)$.  Moreover, $\int_{\partial\Omega}\nabla_{\partial\Omega} \varphi_i\cdot\nabla_{\partial\Omega} \varphi_k d\sigma=\lambda_i\delta_{ik}$ for all $i,k\in\mathbb N$.
\end{theorem}
The proof of Theorem \ref{LBeigen} follows from standard spectral theory for linear operators (see \cite{brezis,daviesbook}) and from the compactness of the embedding $H^1(\partial\Omega)\subset L^2(\partial\Omega)$.



We are now ready to prove Theorem \ref{main-theorem}.

\begin{proof}[Proof of Theorem \ref{main-theorem}]
We start by proving $i)$. Let $u_1,...,u_j$ be the Steklov eigenfunctions associated with $\sigma_1,...,\sigma_j$ normalized such that $\int_{\partial\Omega}u_iu_kd\sigma=\delta_{ik}$, so that $\int_{\Omega}\nabla u_i\cdot\nabla u_k dx=\sigma_i\delta_{ik}$ for all $i,k=1,...,j$. Moreover $\int_{\partial\Omega}u_id\sigma=0$ for all $i=1,...,j$. From the regularity assumptions on $\Omega$, we have that $u_i$ are classical solutions, i.e., $u_i\in C^2(\Omega)\cap C^1(\overline\Omega)$ (see \cite{agmon1}). In particular, ${u_i}_{|_{\partial\Omega}}\in \tilde H^1(\partial\Omega)$ and $\frac{\partial u_i}{\partial\nu}=\sigma_i u$ on $\partial\Omega$, for all $i=1,...,j$. Let $V\subset\tilde H^1(\partial\Omega)$ be the space generated by ${u_1}_{|_{\partial\Omega}},...,{u_j}_{|_{\partial\Omega}}$. Any function $u\in V$ with $\int_{\partial\Omega}u^2d\sigma=1$ can be written as $u=\sum_{i=1}^jc_i{u_i}_{|_{\partial\Omega}}$, where $c=(c_1,...,c_j)\in\mathbb R^j$ is such that $|c|=1$, i.e., $c\in\partial\mathbb B^j$ and $\mathbb B^j$ is the unit ball in $\mathbb R^j$. Moreover $\Delta u=0$ for all $u\in V$. From \eqref{LB_minmax} and \eqref{estimate_tangential} we have
\begin{multline*}
\lambda_j\leq\max_{\substack{0\ne u\in V\\\int_{\partial\Omega}u^2d\sigma=1}}\int_{\partial\Omega}|\nabla_{\partial\Omega} u|^2d\sigma=\max_{\substack{c\in\mathbb B^j\\c=(c_1,...,c_j)}}\int_{\partial\Omega}\left|\nabla_{\partial\Omega}\left(\sum_{i=1}^jc_iu_i\right)\right|^2d\sigma\\
\leq \max_{\substack{c\in\mathbb B^j\\c=(c_1,...,c_j)}}\left(\int_{\partial\Omega}\left(\frac{\partial \left(\sum_{i=1}^jc_iu_i\right)}{\partial\nu}\right)^2d\sigma+{2c_{\Omega}}\left(\int_{\partial\Omega}\left(\frac{\partial \left(\sum_{i=1}^jc_iu_i\right)}{\partial\nu}\right)^2d\sigma\right)^{\frac{1}{2}}\right)\\
=\max_{\substack{c\in\mathbb B^j\\c=(c_1,...,c_j)}}\left(\int_{\partial\Omega}\left(\sum_{i=1}^jc_i\sigma_iu_i\right)^2d\sigma+{2c_{\Omega}}\left(\int_{\partial\Omega}\left(\sum_{i=1}^jc_i\sigma_iu_i\right)^2d\sigma\right)^{\frac{1}{2}}\right)\\
=\max_{\substack{c\in\mathbb B^j\\c=(c_1,...,c_j)}}\left(\sum_{i=1}^jc_i^2\sigma_i^2+{2c_{\Omega}}\left(\sum_{i=1}^jc_i^2\sigma_i^2\right)^{\frac{1}{2}}\right)=\sigma_j^2+{2c_{\Omega}}\sigma_j.
\end{multline*}
This proves $i)$. In an analogous way we prove $ii)$. Let $\varphi_1,...,\varphi_j\in H^1(\partial\Omega)$ be the eigenfunctions associated with the eigenvalues $\lambda_1,...,\lambda_j$ of problem \eqref{Laplace-Beltrami}, normalized such that $\int_{\partial\Omega}\varphi_i\varphi_kd\sigma=\delta_{ik}$ for all $i,k=1,...,j$. Then $\int_{\partial\Omega}\nabla_{\partial\Omega}u_i\cdot\nabla_{\partial\Omega} u_k d\sigma=\lambda_i\delta_{ik}$ for all $i,k=1,...,j$. Moreover $\int_{\partial\Omega}\varphi_id\sigma=0$ for all $i=1,...,j$, thus $\varphi_i\in\tilde H^1(\partial\Omega)$. Now let $\phi_i$, $i=1,...,j$ be the solutions to
\begin{equation}\label{dirichlet}
\begin{cases}
\Delta\phi_i=0, & {\rm in\ }\Omega,\\
\phi_i=\varphi_i, & {\rm on\ }\partial\Omega.
\end{cases}
\end{equation}
It is standard to prove that for all $i=1,...,j$, problem \eqref{dirichlet} admits a unique solution $\phi_i$ which is harmonic inside $\Omega$ and which coincides with $\varphi_i$ on $\partial\Omega$ (see e.g., \cite[Theroem 2.14]{gitr}. From the fact that $\Omega$ is of class $C^2$ and from standard elliptic regularity (see \cite{agmon1}) it follows that $\phi_i\in C^2(\Omega)\cap C^0(\overline\Omega)$. Moreover $\int_{\partial\Omega}{\phi_i}_{|_{\partial\Omega}}d\sigma=\int_{\partial\Omega}\varphi_id\sigma=0$ for all $i=1,...,j$, thus $\phi_i\in \tilde H^1(\Omega)$ for all $i=1,...,j$. Let $W\subset\tilde H^1(\Omega)$ be the space generated by $\phi_1,...\phi_j$. Any function $\phi\in W$ with $\int_{\partial\Omega}\phi^2d\sigma=1$ can be written as $\phi=\sum_{i=1}^jc_i\phi_i$ with $c=(c_1,...,c_j)\in\mathbb B^j$. Moreover $\Delta\phi=0$ for all $\phi\in V$. Thanks to \eqref{estimate_normal} and \eqref{steklov_minmax} we have
\begin{multline*}
\sigma_j\leq\max_{\substack{0\ne \phi\in W\\\int_{\partial\Omega}\phi^2d\sigma=1}}\int_{\Omega}|\nabla \phi|^2dx= \max_{\substack{c\in\mathbb B^j\\c=(c_1,...,c_j)}}\int_{\Omega}\left|\nabla\left(\sum_{i=1}^jc_i\phi_i\right)\right|^2dx\\
\leq \max_{\substack{c\in\mathbb B^j\\c=(c_1,...,c_j)}}\left(\int_{\partial\Omega} \left(\frac{\partial\left(\sum_{i=1}^jc_i\phi_i\right)}{\partial\nu}\right)^2d\sigma\right)^{\frac{1}{2}}\\
\leq {c_{\Omega}}+\left({c_{\Omega}^2}+\max_{\substack{c\in\mathbb B^j\\c=(c_1,...,c_j)}}\int_{\partial\Omega}\left|\nabla_{\partial\Omega}\left(\sum_{i=1}^jc_i\phi_i\right)\right|^2\right)^{\frac{1}{2}}\\
= {c_{\Omega}}+\left( {c_{\Omega}^2}+\max_{\substack{c\in\mathbb B^j\\c=(c_1,...,c_j)}}\int_{\partial\Omega}\left|\nabla_{\partial\Omega}\left(\sum_{i=1}^jc_i\varphi_i\right)\right|^2\right)^{\frac{1}{2}}\\
\leq {c_{\Omega}}+\left( {c_{\Omega}^2}+\max_{\substack{c\in\mathbb B^j\\c=(c_1,...,c_j)}}\sum_{i=1}^jc_i^2\lambda_i\right)^{\frac{1}{2}}
={c_{\Omega}}+\sqrt{ {c_{\Omega}^2}+\lambda_j}.
\end{multline*}
This concludes the proof of $ii)$ and of the theorem.
\end{proof}

Theorem \ref{main-theorem} not only confirms the Weyl asymptotic behavior $\lim_{j\rightarrow\infty}{\sqrt{\lambda_j}}/{\sigma_j}=1$, but says that the difference between the eigenvalues is given at most by a constant independent of $j$.


 By combining \eqref{steklov-laplace-beltrami-ev-comparison-1} with \eqref{laplace-beltrami-ev-weyl-upper-bound} we can now bound the Steklov eigenvalues from above. To this purpose, it is convenient to specify the constants $a_{\partial\Omega}$ and $b_N$ in \eqref{laplace-beltrami-ev-weyl-upper-bound} by recalling the following theorem from \cite{buser}. We  will denote by $Ric_g(M)$ the Ricci curvature tensor of a Riemannian manifold $(M,g)$. Accordingly, $Ric_g(\partial\Omega)$ will denote the Ricci curvature tensor of the submanifold $\partial\Omega$ equipped with the induced Riemannian metric $g$. 
\begin{theorem}\label{buser}
Let $(M,g)$ be a compact Riemannian manifold without boundary of dimension $N$ such that $Ric_g(M)\geq-(N-1)\kappa^2$, $\kappa>0$. Then
\begin{equation}\label{upperbound_LB}
\lambda_j\leq\frac{(N-1)\kappa^2}{4}+c_N\left(\frac{j}{Vol(M)}\right)^{\frac{2}{N}},
\end{equation}
where $c_N>0$ depends only on $N$.
\end{theorem}
From Theorems \ref{main-theorem} and \ref{buser} it immediately follows

\begin{corollary}\label{steklov_upper_bounds_thm}
Let $\Omega$ be a bounded domain of class $C^2$ in $\mathbb R^{N+1}$ such that $\partial\Omega$ has only one connected component. Then for all $j\in\mathbb N$ it holds
\begin{equation}\label{steklov_upper_bounds}
\sigma_j\leq a_{\Omega}+c_N^{\frac{1}{2}}\left(\frac{j}{|\partial\Omega|}\right)^{\frac{1}{N}},
\end{equation}
where $a_{\Omega}>0$ depends on the dimension $N$, on the maximal mean curvature of $\partial\Omega$, on a lower bound of the Ricci curvature of $\partial\Omega$ and on the maximal size of a tubular neighborhood about $\partial\Omega$, and $c_N>0$ is as in Theorem \ref{buser} and depends only on the dimension $N$.

\proof
It suffices just to combine \eqref{upperbound_LB} with the second inequality in \eqref{steklov-laplace-beltrami-ev-comparison-1}. We have
\begin{multline}\label{chainbounds}
\sigma_j\leq{c_{\Omega}}+\sqrt{ {c_{\Omega}^2}+\frac{(N-1)\kappa^2}{4}+c_N\left(\frac{j}{Vol(M)}\right)^{\frac{2}{N-1}}}\\
\leq \left({2c_{\Omega}}+\frac{(N-2)\kappa}{2}\right)+c_N^{\frac{1}{2}}\left(\frac{j}{|\partial\Omega|}\right)^{\frac{1}{N-1}},
\end{multline}
where $\kappa>0$ is such that $Ric_g({\partial\Omega})\geq-(N-2)\kappa^2$. Since $\partial\Omega$ is a  compact submanifold in $\mathbb R^{N+1}$ of class $C^2$ and therefore $Ric_g({\partial\Omega})$ is continuous on $\partial\Omega$, such a $\kappa$ exists finite. From \eqref{bound_mean_curv} and from the proof of Theorem \ref{main-theorem}, we note that $c_{\Omega}$ in \eqref{chainbounds} can be replaced by $\frac{1}{\bar h}+\frac{NH_{\infty}}{2}$. This concludes the proof.
\endproof
\end{corollary}
We conclude this section with some remarks.


\begin{rem}
We remark that in \eqref{steklov_upper_bounds} we have separated the geometry from the asymptotic behavior of the Steklov eigenvalues. We also note that the constant $c_N$ in \eqref{upperbound_LB} (which depends only on the dimension) is not optimal, in the sense that it is strictly greater than the constant appearing in the Weyl's law of $\lambda_j$, as highlighted in \cite{buser}, thus the constant $c_N^{\frac{1}{2}}$ in \eqref{steklov_upper_bounds} is not optimal in this sense as well.
\end{rem}

\begin{rem}
We remark that the constant ${c_{\Omega}}$ in \eqref{steklov_upper_bounds} may become very big when $\Omega$ presents very thin parts (like in the case of dumbell domains), and this can happen also if the curvature remains uniformly bounded (see Figure \ref{F2}). In the case of convex sets, anyway, it is possible to improve the constant in \eqref{steklov-laplace-beltrami-ev-comparison-1}-\eqref{steklov-laplace-beltrami-ev-comparison-2} and therefore the bounds \eqref{steklov_upper_bounds} (see Section \ref{ex}).
\end{rem}

\begin{rem}\label{regularityremark}
We remark that Theorems \ref{laplace-beltrami-ev-asymptotics} and \ref{laplace-beltrami-ev-weyl-upper-bound} are usually stated for the eigenvalues of the Laplace-Beltrami operator on smooth Riemannian manifolds. Actually, it is sufficient that $\partial\Omega$ is a manifold of class $C^2$ for \eqref{laplace-beltrami-ev-asymptotics} and \eqref{laplace-beltrami-ev-weyl-upper-bound} to hold. In fact we can approximate $\partial\Omega$ with a sequence $\partial\Omega_{\varepsilon}$ of $C^{\infty}$ submanifolds such that $\partial\Omega=\psi_{\varepsilon}(\partial\Omega_{\varepsilon})$, where $\psi_{\varepsilon}$ is a diffeomorphism of class $C^2$ and $\|{\rm Id}-\psi_{\varepsilon}\|_{C^2(\partial\Omega_{\varepsilon})},\|{\rm Id}-\psi_{\varepsilon}^{(-1)}\|_{C^2(\partial\Omega)}\leq\varepsilon$. This follows from standard approximation of $C^k$ functions by $C^{\infty}$ (or analytic) functions (see \cite{whitneyapprox}). We also refer to \cite[Sec. 4.4]{pruss} for a more detailed construction of the approximating boundaries $\partial\Omega_{\varepsilon}$. It is then standard to prove that the eigenvalues of the Laplace-Beltrami operator on $\partial\Omega_{\varepsilon}$ pointwise converge the eigenvalues of the Laplace-Beltrami operator on $\partial\Omega$. This immediately follows from the min-max characterization of the eigenvalues \eqref{LB_minmax} (we also refer to \cite{laproeurasian,kalamata} for stability and continuity results for the eigenvalues of elliptic operators upon perturbations of some parameters entering the equation and to \cite{burenkovlamberti2007,burenkovlamberti2008,burenkovlamberti2012} and to the references therein for spectral stability results for eigenvalues upon perturbation of the domain). We also refer to \cite{colbois_conv,conv_sp_str} and to the references therein for more detailed information on the convergence of Riemannian manifolds and the convergence of the corresponding spectra of the Laplacian.  

Moreover, from the fact that $\|{\rm Id}-\psi_{\varepsilon}\|_{C^2(\partial\Omega_{\varepsilon})},\|{\rm Id}-\psi_{\varepsilon}^{(-1)}\|_{C^2(\partial\Omega)}\leq\varepsilon$, it follows that $|\partial\Omega_{\varepsilon}|\rightarrow|\partial\Omega|$ and if $\kappa>0$ is such that $Ric_{g}(\partial\Omega)\geq -(N-1)\kappa^2$, then there exists a sequence $\kappa_{\varepsilon}$ with $\kappa_{\varepsilon}\rightarrow\kappa$ as $\varepsilon\rightarrow 0$ such that $Ric_{g_{\varepsilon}}(\partial\Omega_{\varepsilon})\geq -(N-1)\kappa_{\varepsilon}^2$. Hence \eqref{laplace-beltrami-ev-asymptotics} and \eqref{laplace-beltrami-ev-weyl-upper-bound} hold if $\Omega$ is of class $C^2$.
\end{rem}



\section{Examples: convex domains and balls}\label{ex}
In this section we improve the constant in \eqref{steklov-laplace-beltrami-ev-comparison-1}-\eqref{steklov-laplace-beltrami-ev-comparison-2} and the bounds \eqref{steklov_upper_bounds} in the case when $\Omega$ is a convex and bounded domain of class $C^2$ and show that the corresponding estimates become sharp when $\Omega$ is a ball.

\subsection{Convex domains}
Let $\Omega$ be a convex domain of class $C^2$ in $\mathbb R^{N+1}$. It is well-known that in this case $\kappa_i(x)\geq 0$ for all $i=1,...,N$ and for all $x\in\partial\Omega$. Moreover Theorem \ref{tubular} holds for any ${h}\in]0,{1}/K_{\infty}[$ (see also \eqref{KKKK} for the definition of $K_{\infty}$). This follows from Blaschke's Rolling Theorem for $C^2$ convex domains (see \cite{brooks,delgado,how,kou}) and from \cite[Lemma 14.16]{gitr}. 

From \eqref{maxemin} and from the fact that $0\leq\rho_i(x)\leq 1$ for all $x\in\omega_h$ and $i=1,...,N+1$ (see also Remark \ref{convex}), it follows that
\begin{equation}\label{estimate_conv}
-\int_{\Omega}|\nabla v|^2 dx\leq\int_{\omega_{{h}}}|\nabla v|^2\Delta{\eta}-2(D^2{\eta}\cdot\nabla v)\cdot\nabla v dx\leq N\int_{\Omega}|\nabla v|^2 dx.
\end{equation}
Then, by following the same lines of the proof of Theorems \ref{main-theorem} and \ref{equivalence} and choosing $\bar h=1/K_{\infty}$, it is straightforward to prove the following:
\begin{theorem}
Let $\Omega$ be a bounded and convex domain of class $C^2$ in $\mathbb R^{N+1}$. Let $\sigma_j$ and $\lambda_j$, $j\in\mathbb N$, denote the eigenvalues of problems \eqref{Steklov} and \eqref{Laplace-Beltrami} respectively. Let $K_{\infty}$ be defined by \eqref{KKKK}. Then
\begin{enumerate}[i)]
\item
\begin{equation}\label{LB_bounds_conv}
\lambda_j\leq\sigma_j^2+NK_{\infty}\sigma_j;
\end{equation}
\item
\begin{equation*}
\sigma_j\leq \frac{K_{\infty}}{2}+\sqrt{\frac{K_{\infty}^2}{4}+\lambda_j}.
\end{equation*}
\end{enumerate}
\end{theorem}


We note that when $\Omega$ is a bounded and convex domain of class $C^2$, $Ric_g(\partial\Omega)\geq 0$. Accordingly, as a consequence of Theorem \ref{buser}, we have the following:
\begin{corollary}
Let $\Omega$ be a bounded and convex domain of class $C^2$ in $\mathbb R^{N+1}$. Let $\sigma_j$ and $\lambda_j$, $j\in\mathbb N$, denote the eigenvalues of problem \eqref{Steklov} and \eqref{Laplace-Beltrami} respectively. Let $K_{\infty}$ be defined by \eqref{KKKK}. Then
\begin{equation*}
\sigma_j\leq K_{\infty}+c_N^{\frac{1}{2}}\left(\frac{j}{|\partial\Omega|}\right)^{\frac{1}{N}}.
\end{equation*}
\end{corollary}
We note that the geometry of the set enters in the estimate only by means of the maximum of the principal curvatures.
\begin{rem}
Suppose that $\Omega$ is a convex and bounded domain of class $C^2$ such that $\left(\sum_{i=1}^{N}\rho_i(x)-1\right)\geq 0$ for all $x\in\omega_{{h}}$. Then by \eqref{estimate_conv} and by the same arguments in the proof of Theorems \ref{main-theorem} and \ref{equivalence} we have
$$
\sigma_j\leq c_N^{\frac{1}{2}}\left(\frac{j}{|\partial\Omega|}\right)^{\frac{1}{N}}
$$
for all $j\in\mathbb N$.
\end{rem}

\subsection{Balls}\label{balls}
Let $\Omega$ be a ball of radius $R$ in $\mathbb R^{N+1}$. We can suppose without loss of generality that it is centered at the origin. We are allowed to take ${h}=R-\delta$ for all $\delta\in]0,R[$ through Sections \ref{sec:pre},\ref{sec:har} and \ref{sec:comp}. By letting $\delta\rightarrow 0$, the expression for the vector field given by $F$ in \eqref{field} simplifies to $F(x)=x$ for all $x\in\Omega$. We use $F(x)=x$ in \eqref{pohozaev} and we obtain that for all $v\in H^2(\Omega)$ with $\Delta v=0$ in $\Omega$ it holds:
\begin{enumerate}[i)]
\item
\begin{equation}\label{Ball1}
\int_{\partial\Omega}|\nabla_{\partial\Omega}v|^2d\sigma=\int_{\partial\Omega}\left(\frac{\partial v}{\partial\nu}\right)^2d\sigma+\frac{N-1}{R}\int_{\Omega}|\nabla v|^2 d\sigma;
\end{equation}
\item
\begin{equation}\label{Ball2}
\int_{\partial\Omega}\left(\frac{\partial v}{\partial\nu}\right)^2d\sigma=\int_{\partial\Omega}|\nabla_{\partial\Omega}v|^2d\sigma-\frac{N-1}{R}\int_{\Omega}|\nabla v|^2 d\sigma.
\end{equation}
\end{enumerate}
We find then that
\begin{enumerate}[i)]
\item
\begin{equation}\label{LB_bounds_conv_B}
\lambda_j\leq\sigma_j^2+\frac{(N-1)}{R}\sigma_j;
\end{equation}
\item
\begin{equation}\label{steklov_bounds_conv_B}
\sigma_j\leq\sqrt{\frac{(N-1)^2}{4R^2}+\lambda_j}-\frac{N-1}{2R}.
\end{equation}
\end{enumerate}
Inequality \ref{LB_bounds_conv_B} follows immediately from \eqref{Ball1} by the same arguments as in the proof of Theorems \ref{equivalence} and \ref{main-theorem}. For \eqref{steklov_bounds_conv_B}, we note that if $\varphi_j\in H^1(\partial\Omega)$ is an eigenfunction associated with the eigenvalue $\lambda_j$ of \eqref{Laplace-Beltrami} and if we denote by $\phi_j$ the unique solution to \eqref{dirichlet}, then from \eqref{Ball2} we have
\begin{multline*}
0=\lambda_j-\int_{\partial\Omega}\left(\frac{\partial \phi_j}{\partial\nu}\right)^2d\sigma-\frac{N-1}{R}\int_{\Omega}|\nabla \phi_j|^2 d\sigma\\
\leq \lambda_j-\left(\int_{\Omega}|\nabla\phi_j|^2dx\right)^2-\frac{N-1}{R}\int_{\Omega}|\nabla\phi_j|^2 dx.
\end{multline*}
This in particular implies
\begin{equation*}
\int_{\Omega}|\nabla\phi_j|^2dx\leq\sqrt{\frac{(N-1)^2}{4R^2}+\lambda_j}-\frac{N-1}{2}
\end{equation*}
and therefore, by the min-max principle \eqref{steklov_minmax}, the validity of \eqref{steklov_bounds_conv_B}. Combining \eqref{LB_bounds_conv_B} with \eqref{steklov_bounds_conv_B} we immediately obtain the exact relation among the eigenvalues of problems \eqref{Steklov} and \eqref{Laplace-Beltrami} on $\Omega$ and $\partial\Omega$ respectively, without knowing explicitly the eigenvalues. Namely we have the following:
\begin{equation}\label{equal}
\lambda_j=\sigma_j^2+\frac{(N-1)}{R}\sigma_j.
\end{equation}

For the reader convenience, we briefly recall the explicit formulas for the Laplacian eigenvalues on $\partial\Omega$ and the Steklov eigenvalues on $\Omega$. An eigenvalue $\lambda$ of the Laplace-Beltrami operator on $\partial\Omega$ is of the form $\lambda=\frac{l(l+N-1)}{R^2}$, with $l\in\mathbb N$. Let us denote by $H_l$ a spherical harmonic of degree $l$ in $\mathbb R^{N+1}$. An eigenfunction associated with the eigenvalue $\frac{l(l+N-1)}{R^2}$ is of the form $H_l(x/R)$, $x\in\partial\Omega$. Hence the multiplicity of the eigenvalue $\lambda=\frac{l(l+N-1)}{R^2}$ equals the dimension $d_l$ of the space of the spherical harmonics of degree $l$ in $\mathbb R^{N+1}$, namely $d_l=(2l+N-1)\frac{(l+N-2)!}{l!(N-1)!}$. On the other hand, a Steklov eigenvalue $\sigma$ on $\Omega$ is of the form $\sigma=\frac{l}{R}$ with $l\in\mathbb N$. The corresponding eigenfunctions are the restriction to $\Omega$ of the harmonic polynomials on $\mathbb R^{N+1}$ of degree $l$. Clearly the eigenvalues $\frac{l(l+N-1)}{R^2}$ and $\frac{l}{R}$ have the same multiplicity $d_l$. It is now immediate to see that formula \eqref{equal} holds true.



\subsection{A further example: a bounded and convex domain of class $C^{1,1}$}.

Throughout the paper we have considered bounded domains of class $C^2$. This is a sufficient condition to ensure the validity of Theorems \ref{tubular0} and \ref{tubular}. Actually, Theorems \ref{tubular0} and \ref{tubular} may hold also under lower regularity assumptions on $\Omega$. It is known that the existence of a tubular neighborhood $\omega_h$ of $\partial\Omega$ as in Theorem \ref{tubular0} implies that the distance function from $\partial\Omega$ is a function of class $C^{1,1}$ on $\omega_h$. We refer to \cite[Ch.7]{shapes} for a more detailed discussion on sets of positive reach.

We construct now a convex subset $\Omega$ of $\mathbb R^3$ of class $C^{1,1}$ such that the set of points in $\Omega$ where the distance function is not differentiable has zero Lebesgue measure (in particular, it is a segment) and such that $\left(\sum_{i=1}^{3}\rho_i(x)-1\right)\geq 0$.
Let $x=(x_1,x_2,x_3)$ denotes an element of $\mathbb R^3$. Let $L,R>0$ be fixed real numbers. Let $x_0^+:=(0,0,L)$ and $x_0^-:=(0,0,-L)$. Let $\Omega\subset\mathbb R^3$ be defined by
\begin{equation*}
\Omega:=\Omega_1\cup\Omega_2\cup\Omega_3,
\end{equation*}
where
\begin{equation*}
\Omega_1:=\left\{x\in\mathbb R^3:|x-x_0^+|<R\right\}\cap \left\{x\in\mathbb R^3:x_3\geq L\right\},
\end{equation*}
\begin{equation*}
\Omega_2:=\left\{x\in\mathbb R^3:x_1^2+x_2^2<R^2\right\}\cap \left\{x\in\mathbb R^3:-L\leq x_3\leq L\right\}
\end{equation*}
and
\begin{equation*}
\Omega_3:=\left\{x\in\mathbb R^3:|x-x_0^-|<R\right\}\cap \left\{x\in\mathbb R^3:x_3\leq -L\right\}.
\end{equation*}
By construction $\Omega$ is of class $C^{1,1}$ but it is not of class $C^2$. Moreover it is convex. We note that we can take ${h}=R-\delta$ for all $\delta\in]0,R[$. Hence, as in the case of the ball, we can take in \eqref{pohozaev} the vector field defined by
\begin{equation*}
F(x)=\begin{cases}
x-x_0^+, & {\rm if\ }x\in\Omega_1,\\
(x_1,x_2,0), & {\rm if\ }x\in\Omega_2,\\
x-x_0^+, & {\rm if\ }x\in\Omega_3.
\end{cases}
\end{equation*}
By construction, $F$ is a Lipschitz vector field. Standard computations show that
$$
\rho_i(x)=1,
$$
for all $x\in\Omega_1\cup\Omega_3$ and for $i=1,2,3$ and
$$
\lambda_1(x)=0,\ \ \ \lambda_2(x)=\lambda_3(x)=1,
$$
for all $x\in\Omega_2$. Hence $\left(\sum_{i=1}^{2}\rho_i(x)-1\right)\geq 0$ for all $x\in\Omega$. Then for the Steklov eigenvalues $\sigma_j$ on $\Omega$ we have $\sigma_j\leq c_2^{\frac{1}{2}}\left(\frac{j}{|\partial\Omega|}\right)^{\frac{1}{2}}$.

\section{Proof of Theorem \ref{Riesz-mean-theorem}}\label{riesz}

In this section we prove Theorem \ref{Riesz-mean-theorem}, namely we prove asymptotically sharp upper bounds for Riesz means of Steklov eigenvalues. As a consequence we provide asymptotically sharp upper bounds for the trace of the Steklov heat kernel and lower bounds for Steklov eigenvalues. 

\begin{proof}[Proof of Theorem \ref{Riesz-mean-theorem}]
For the Laplacian eigenvalues $\lambda_i$ on $\partial \Omega$ the following asymptotically sharp inequality has been shown in \cite{HaSt11}:
\begin{equation}\label{Laplacian-ev-HS-weyl-estimate}
  \sum_{j=0}^{\infty}(z-\lambda_j)_{+}^2\leq \frac{8}{(N+2)(N+4)}\,(2\pi)^{-N}B_N|\partial\Omega|(z+z_0)^{2+\frac{N}{2}}
\end{equation}
where $\displaystyle z_0:=\frac{N^2}{4}\,H_{\infty}^2$ and $H_{\infty}$ is given by \eqref{HHH}. We note that $z_0\leq c_{\Omega}^2$. It follows from the first inequality of \eqref{steklov-laplace-beltrami-ev-comparison-1} of Theorem \ref{main-theorem} that
\begin{equation}\label{laplacian-steklov-Riesz-mean-ineq}
  \sum_{j=0}^{\infty}(z-\lambda_j)_{+}\geq \sum_{j=0}^{\infty}(z-\sigma_j^2-2c_{\Omega}\sigma_j)_{+}.
\end{equation}
Defining a new variable $\zeta$ by $\displaystyle \zeta:= \sqrt{z+c_{\Omega}^2}-c_{\Omega}$ it is easily shown that \eqref{laplacian-steklov-Riesz-mean-ineq} is equivalent to
\begin{equation*}
  \sum_{j=0}^{\infty}(\zeta^2+2c_{\Omega}\zeta-\lambda_j)_{+}\geq 2(\zeta+c_{\Omega})\sum_{j=0}^{\infty}(\zeta-\sigma_j)_{+}-\sum_{j=0}^{\infty}(\zeta-\sigma_j)_{+}^2
\end{equation*}
and therefore it is equivalent to the differential inequality
\begin{equation}\label{laplacian-steklov-Riesz-mean-diff-ineq}
  \frac{d}{d\zeta}\frac{\sum_{j=0}^{\infty}(\zeta-\sigma_j)_{+}^2}{\zeta+c_{\Omega}}\leq \frac{\sum_{j=0}^{\infty}(\zeta^2+2c_{\Omega}\zeta-\lambda_j)_{+}}{(\zeta+c_{\Omega})^2}.
\end{equation}
Integrating the differential inequality \eqref{laplacian-steklov-Riesz-mean-diff-ineq} between $0$ and $\zeta$ and performing an integration by parts on the right-hand side of the resulting inequality, we obtain
\begin{equation*}
  \frac{\sum_{j=0}^{\infty}(\zeta-\sigma_j)_{+}^2}{\zeta+c_{\Omega}}\leq \frac{\sum_{j=0}^{\infty}(\zeta^2+2c_{\Omega}\zeta-\lambda_j)_{+}}{4(\zeta+c_{\Omega})^3}
  +\frac{3}{4}\int_0^{\zeta}\frac{\sum_{j=0}^{\infty}(s^2+2c_{\Omega}s-\lambda_j)^2_{+}}{(s+c_{\Omega})^4}\,ds.
\end{equation*}
We apply estimate \eqref{Laplacian-ev-HS-weyl-estimate}, replace $z_0$ by $c_{\Omega}^2$ and compute the resulting integral. We get the inequality
\begin{equation*}
  \sum_{j=0}^{\infty}(\zeta-\sigma_j)_{+}^2\leq \frac{2}{(N+2)(N+4)}\,(2\pi)^{-N}B_N|\partial\Omega|(\zeta+c_{\Omega})^{1+N}\big(1+\frac{3}{N+1}\big)
\end{equation*}
which proves the claim.
\end{proof}

Laplace transforming inequality \eqref{Steklov-ev-Riesz-mean-upper-bound} of Theorem \ref{Riesz-mean-theorem} yields the following upper bound on the trace of the heat kernel for the Steklov operator:
\begin{corollary}\label{Steklov-heat-trace-upper-bound-cor}
Let $\Omega$ be a bounded domain of class $C^2$ in $\mathbb R^{N+1}$ such that $\partial\Omega$ has only one connected component. Then
\begin{equation}\label{Steklov-heat-trace-upper-bound}
  \sum_{j=0}^{\infty}e^{-\sigma_jt}\leq \frac{1}{(N+1)(N+2)}\,(2\pi)^{-N}B_N|\partial\Omega|t^{-N}e^{c_{\Omega}t}\Gamma(N+3,c_{\Omega}t)
\end{equation}
for all $t>0$, where $\displaystyle \Gamma(a,b)=\int_b^{\infty}t^{a-1}e^{-t}\,dt$ denotes the incomplete Gamma function. 
\end{corollary}
The estimate is sharp as $t$ tends to zero since \eqref{Steklov-heat-trace-upper-bound} implies the exact bound
\begin{equation*}
  \underset{t\rightarrow 0_{+}}{\limsup}\,t^N\sum_{j=0}^{\infty}e^{-\sigma_jt}\leq (2\pi)^{-N}B_N\Gamma(N+1)|\partial\Omega|.
\end{equation*}
From \eqref{Steklov-heat-trace-upper-bound} we immediately obtain Weyl-type lower bounds on Steklov eigenvalues. Since $\displaystyle (j+1)e^{-\sigma_jt}\leq \sum_{k=0}^{\infty}e^{-\sigma_kt}$ for all $j\in\mathbb N$ and $\Gamma(N+3,c_{\Omega}t)\leq \Gamma(N+3)$ we get from \eqref{Steklov-heat-trace-upper-bound} after optimizing with respect to $t$ the following:
\begin{corollary}\label{Steklov-lower-bound} 
Let $\Omega$ be a bounded domain of class $C^2$ in $\mathbb R^{N+1}$ such that $\partial\Omega$ has only one connected component. Then for all $j\in\mathbb N$:
\begin{equation*}
  \sigma_j\geq r_N2\pi B_N^{-1/N}\left(\frac{j+1}{|\partial\Omega|}\right)^{\frac{1}{N}}-c_{\Omega}
\end{equation*}
with $\displaystyle r_N= \frac{N}{e \Gamma(N+1)^{1/N}}\leq 1$.
\end{corollary}

{\bf Acknowledgments.} The first author is member of the Gruppo Nazionale per l'Analisi Matematica, la Probabilit\`{a} e le loro Applicazioni (GNAMPA) of the Istituto Nazionale di Alta Matematica (INdAM).

\bibliography{bibliography}{}
\bibliographystyle{abbrv}

\end{document}